\documentclass{amsart}
\usepackage[utf8]{inputenc}
\usepackage{amsfonts}
\usepackage{amssymb}
\usepackage{amsmath}
\usepackage{amsthm}
\usepackage{cleveref}
\usepackage{mathtools}

\usepackage{tikz-cd}
\usepackage{physics}
\usepackage{tikz}
\usepackage{mathdots}
\usepackage{yhmath}
\usepackage{cancel}
\usepackage{color}
\usepackage{siunitx}
\usepackage{array}
\usepackage{multirow}
\usepackage{tabularx}
\usepackage{extarrows}
\usepackage{booktabs}
\usepackage{adjustbox}

\usepackage{biblatex}
\usetikzlibrary{fadings}
\usetikzlibrary{patterns}
\usetikzlibrary{shadows.blur}
\usetikzlibrary{shapes}

\newtheorem{theorem}{Theorem}[section]

\newtheorem{lemma}[theorem]{Lemma}
\theoremstyle{definition}
\newtheorem{definition}[theorem]{Definition}
\theoremstyle{proposition}
\newtheorem{proposition}[theorem]{Proposition}
\newtheorem{remark}[theorem]{Remark}
\usepackage{thm-restate}

\newcommand{\lr}[1]{\langle #1 \rangle}

\begin{document}
\title{If a non-RF hyperbolic group exists, then a rigid one exists}
\author{Xuzhi Tang}
\date{}

\begin{abstract}
If there is a non-residually finite hyperbolic group, then there is a non-residually finite rigid hyperbolic group.
\end{abstract}
\maketitle
\section{Introduction}
Whether all hyperbolic groups are residually finite is a famous open problem in geometric group theory \cite[Section~5.3,Page~141]{Gromov}. There has been much work done in this direction. Kapovich and Wise showed that all hyperbolic groups are residually finite if and only if every nontrivial hyperbolic group has a proper finite index subgroup \cite[Theorem~1.2]{KW00}. Agol, Groves, and Manning showed that all hyperbolic groups are residually finite if and only if all hyperbolic groups are QCERF \cite[Theorem~0.1]{QCERF}. In this paper, we show that the existence of a non-residually finite hyperbolic group implies the existence of a non-residually finite rigid hyperbolic group.\\
\begin{definition}
	A group $G$ is \textbf{residually finite} if for every nontrivial element $g$, there is a homomorphism $\phi:G \to F$ with $F$ finite and $\phi(g)$ nontrivial.
\end{definition}

\begin{definition}
	A hyperbolic group $G$ is \textbf{rigid} if $G$ does not split as a nontrivial graph of groups with finite and/or virtually cyclic edge groups.
\end{definition}
\begin{restatable}{theorem}{Main}\label{main}
	If there is a non-residually finite hyperbolic group, then there is a non-residually finite rigid hyperbolic group.
\end{restatable}
	For a hyperbolic group, \Cref{RF} and Dunwoody's accessibility (\cite[Theorem~5.1]{Dun85}, \Cref{access}) reduces our argument to the one-ended case. Given $g$ in a one-ended hyperbolic group $G$, \Cref{tech} and \Cref{rigid} (\cite[Theorem~1.2]{Das22}) together with the hypothesis of \Cref{main} gives us a homomorphism that doesn't kill $g$ to a group satisfying the conditions of \Cref{RF}, hence residually finite.  See Section 3 for the detailed proof of \Cref{main}.\par 
	\Cref{tech} is the main technical theorem of this paper. See Section 2 for definitions of the terms involved.
\begin{restatable}{theorem}{Tech}\label{tech}
Let $G$ be a one-ended hyperbolic group that is not finite-by-cocompact Fuchsian, and $(\Gamma,G)$ be its canonical JSJ-decomposition. Let $\mathcal{P}=\{P_1,...,P_n\}$ be the peripheral structure induced by the edge groups of $\Gamma$. Then, for any nontrivial element $g$, the following holds for a cofinal family of fillings $\bar{G} = G(N_1,...,N_n)$ with respect to $\mathcal{P}$:
\begin{enumerate}
	\item[1)] the image of $g$ under the quotient map is nontrivial,
	\item[2)] $\bar{G}$ admits a splitting with the same graph $\Gamma$ and finite edge groups,
	\item[3)] if $G_v$ is flexible, then $\bar{G}_v$ is residually finite.
\end{enumerate}
\end{restatable}
	Property 1 is a standard consequence of Dehn filling. There are plenty of quotients satisfying property 2 (\Cref{quotient-graph}). For property 3, since we are using finite index fillings, the maximal two-ended vertex groups have finite image. Then, we show that the image of each QH vertex group (the other kind of flexible vertex groups) in a smaller family of quotients are residually finite (\Cref{orbi-shift}). Finally, it is straightforward to see that these families have non-empty intersection (\Cref{QH}). See Section 3 for the detailed proof of \Cref{tech}.
\begin{remark}	
	Ashot Minasyan pointed out a simple proof to a stronger fact: If there is a non-residually finite hyperbolic group, then there is a non-residually finite hyperbolic group with Serre's property FA.\par 
	Suppose that there is a non-residually finite hyperbolic group, then there is a hyperbolic group $G_1$ without proper finite index subgroups (\cite[Theorem~1.2]{KW00}), and it has to be non-elementary. Let $G_2$ be a non-elementary hyperbolic group with Serre's property FA. By \cite[Theorem~2]{Ol93}, $G_1$ and $G_2$ have a common non-elementary quotient $G$ that is hyperbolic. $G$ is a non-residually finite hyperbolic group with Serre's property FA.
\end{remark}
	\section*{Acknowledgements}
	The author would like to thank his advisor, Daniel Groves, for proposing this project and guiding the author towards its completion, as well as Aaron Messerla for helpful comments on expositions and references. Thanks to Ashot Minasyan for pointing out a simple proof to a stronger fact.
\section{Relatively Hyperbolic Dehn Filling and JSJ Decomposition}
\begin{definition}\cite[Chapter~III.H,Definition~1.1]{MNPC}
A geodesic metric space is \textbf{$\delta$-hyperbolic} if for each geodesic triangle, each of its sides is contained in the $\delta$-neighborhood of the other two sides.\\
A geodesic metric space is \textbf{hyperbolic} if it is $\delta$-hyperbolic for some $\delta>0$.\\
A finitely generated group $G$ is \textbf{hyperbolic} if its Cayley graph with respect to some finite generating set under the graph metric is hyperbolic.
\end{definition}
\begin{definition}\cite[Chapter~III.$\Gamma$,Definition~3.4]{MNPC}
	Let $X$ be a geodesic metric space and $Y \subseteq X$. The subspace $Y$ is \textbf{quasi-convex} in $X$ if there is $K>0$ such that for any $x,y \in Y$, any geodesic, in $X$, between $x$ and $y$ is in the $K$-neighbourhood of $Y$.\\
	Let $H \leq G$ be a subgroup of a finitely generated $G$. Fix a Cayley graph $C$ of $G$ induced by a finite generating set with graph metric. \\
	The subgroup $H$ is \textbf{quasi-convex} in $G$ if the subset of the vertex set of $C$ given by $H$ is quasi-convex in $C$.
\end{definition}
\begin{definition}
	Let $\mathcal{H}$ be a collection of subgroups of $G$. The collection $\mathcal{H}$ is \textbf{almost malnormal} if for any $H_1,H_2 \in \mathcal{H}$, and any $g \in G, |gH_1g^{-1} \bigcap H_2| = \infty$ implies $H_1 = H_2 \text{ and } g \in H_1$.
\end{definition}
\begin{definition}
	Given a hyperbolic group $G$ and a collection of subgroups $\mathcal{H}=\{H_1,...,H_n\}$, an n-tuple of \textbf{filling kernels} is a choice of a normal subgroup $N_i$ of $H_i$ for each $i$.\\ 
	A \textbf{filling} $\bar{G} = G(N_1,...,N_n)$ of $G$ with respect to $\mathcal{H}$ with filling kernels $(N_1,...,N_n)$ is the quotient group given by $\bar{G} = G/\lr{\lr{N_1,...,N_n}}$.\\
	Given a finite subset $B \subseteq G \setminus \{1\}$, let $C_B= \{(N_1,...,N_n) \mid N_i \trianglelefteq H_i, \forall i,N_i \bigcap B = \emptyset\}$. A property $Q$ holds for \textbf{sufficiently long fillings} of $G$ if there is a finite $B \subseteq G \setminus \{1\}$ such that $Q$ holds for $G(N_1,...,N_n)$ with any filling kernels $(N_1,...,N_n) \in C_B$. When we add \textbf{finite index} to modify any notion of fillings (e.g. finite index fillings), we mean that $[H_i:N_i]<\infty$ for each $i$. \\ 
	A \textbf{cofinal family of finite index filling kernels} is a set $S$ of finite index filling kernels so that for any finite $B \subseteq G \setminus \{1\}$, the intersection $S \bigcap C_B$ is non-empty. A \textbf{cofinal family of finite index fillings} are the fillings obtained from using the filling kernels from a cofinal family of finite index filling kernels.
\end{definition}
We make use of the following specialized relative hyperbolic Dehn filling theorem. See \cite[Theorem~2.6]{AGM14} to see why \cite[Theorem~1.1]{OsinDF} specializes in the following way. 
\begin{theorem}[Special Case of {\cite[Theorem~1.1]{OsinDF}}]\label{fill}
	Let $G$ be a hyperbolic group, $\mathcal{H}=\{H_1,...,H_n\}$ an almost malnormal collection of quasiconvex subgroups, and $F \subseteq G$ a finite subset. Then, for sufficiently long finite-index fillings, the following holds:
\begin{itemize}
	\item $\bar{G} = G(N_1,...,N_n)=G/\lr{\lr{N_1,...,N_n}}$ is hyperbolic
	\item $F$ injects into $\bar{G}$ under the quotient map,
	\item for each $i$, the induced map from inclusion $i_*:H_i/N_i \to \bar{G}$ is an embedding
\end{itemize}
\end{theorem}
We utilize graph of groups heavily, but we refer the reader to \cite{Trees} for detailed definitions and basic structure theorems. The following propositions are important in our arguments.
\begin{proposition}[{\cite[Chapter~2,Proposition~12]{Trees}}]\label{RF}
The fundamental group of a finite graph of groups with finite edge groups and residually finite vertex groups is residually finite.
\end{proposition}
\begin{proposition}[{\cite[Proposition~1.2]{Bow98}}]\label{quasi}
	Let $G$ be a hyperbolic group with a graph of groups decomposition where all edge groups are quasi-convex, then all vertex groups are quasi-convex (hence hyperbolic).
\end{proposition}
\begin{definition}
	An \textbf{elementary splitting} of a hyperbolic group $G$ is a graph of groups decomposition of $G$ with finite or two-ended edge groups. 
\end{definition}
	The precise definition of the number of ends of a group is not essential in our arguments here. The reader can refer to \cite[Chapter~I.8, Definition~8.30]{MNPC}. The essential part is the following fact.
\begin{proposition}[{\cite[Theorem~8.32]{MNPC}}]\label{ends}
	Let $G$ be a finitely generated group. Then, the number of ends of $G$ is 0, 1, 2, or $\infty$. $G$ has zero ends if and only if $G$ is finite. $G$ has two ends if and only if $G$ is virtually infinite cyclic.
\end{proposition}
\begin{definition}
	Let $(\Gamma,G)$ be a graph of groups. Let $v$ be a vertex of $\Gamma$ and $G_v$ the associated vertex group. Let \textbf{$Inc|_v$} or \textbf{$Inc|_{G_v}$} denote the collection of all edge groups for edges incident to $v$.
\end{definition}
	The reader can jump to the proof of \Cref{main} at the end of the paper to see why we are only worried about one-ended hyperbolic groups here. For a one-ended hyperbolic group $G$, we have a special splitting (the canonical JSJ decomposition) $(\Gamma,G)$ with all edge groups 2-ended by \cite[Theorem~6.2]{Bow98}. We follow the construction of \cite[Section~3.1]{QCERF} to get an induced peripheral structure from the edge groups. \begin{definition}
	Given a hyperbolic group $G$ and a finite collection $\mathcal{H} = \{H_1,...,H_k\}$ of 2-ended subgroups, we get an \textbf{induced peripheral structure} $\mathcal{P}$ by replacing each 2-ended subgroup with its commensurator in $G$ and then choosing one subgroup per conjugacy class. 
\end{definition} 
\begin{remark}
	In our situation, we only use two-ended subgroups to induce peripheral structures. Infinite cyclic subgroups of hyperbolic groups are quasi-convex by \cite[Chapter~III.$\Gamma$,Corollary~3.10]{MNPC}, so two-ended subgroups are also quasi-convex. Thus, we can follow the procedures of \cite[Section~3.1]{QCERF} for each two-ended subgroup. If $H$ is one of the two-ended subgroups, the groups in the malnormal core described in the procedure of \cite[Section~3.1]{QCERF} are infinite subgroups of $H$. Since $H$ is two-ended, any infinite subgroup has finite index in it, so when we take commensurators (in $G$) of the groups in the malnormal core, all we get is the commensurator of $H$. So, our definition of induced peripheral structure agrees with \cite{QCERF} in our situation.
\end{remark}
\begin{remark}	
	In a hyperbolic group, the commensurator of a two-ended subgroup $H$ is the maximal two-ended subgroup containing $H$. In particular, each of our two-ended subgroups $H$ has finite index in its commensurator.
\end{remark}
	We would like to apply \Cref{fill}, so we need the peripheral structure to be an almost malnormal collection of quasi-convex subgroups.
\begin{lemma}\label{almost}
	Let $G$ be a hyperbolic group and $\mathcal{H}$ a finite collection of two-ended subgroups. Then, the peripheral structure $\mathcal{P}$ induced by $\mathcal{H}$ is an almost malnormal collection of quasi-convex subgroups.
\end{lemma}
\begin{proof} 
	Suppose $|gP_1g^{-1} \bigcap P_2| = \infty$. Say $P_1$ is the commensurator of $H_1$ and $P_2$ is the commensurator of $H_2$. Then, $gH_1g^{-1} \bigcap H_2$ is a finite index subgroup in both $gH_1g^{-1}$ and $H_2$. Take an element $p \in P_2$ and consider the following subgroup,
\begin{align*}
	& N:=(pgH_1g^{-1}p^{-1} \bigcap gH_1g^{-1}) \bigcap (pH_2p^{-1} \bigcap H_2)\\
	 & = p(gH_1g^{-1} \bigcap H_2)p^{-1} \bigcap (gH_1g^{-1} \bigcap H_2)
\end{align*}
	 Now, $p(gH_1g^{-1} \bigcap H_2)p^{-1}$ is a finite index subgroup of $pH_2p^{-1}$. So, $N$ has finite index in $pH_2p^{-1} \bigcap gH_1g^{-1} \bigcap H_2$ which in turn has finite index in $pH_2p^{-1} \bigcap H_2$. By the definition of commensurator, $pH_2p^{-1} \bigcap H_2$ has finite index in $H_2$. So, $N$ has finite index in $H_2$. This means $N$ is infinite and so is $pgH_1g^{-1}p^{-1} \bigcap gH_1g^{-1}$. Since $gH_1g^{-1}$ and $pgH_1g^{-1}p^{-1}$ are two-ended, the intersection $pgH_1g^{-1}p^{-1} \bigcap gH_1g^{-1}$ has finite index in both of them. Therefore, $p$ is in the commensurator of $gH_1g^{-1}$, which is $gP_1g^{-1}$. A symmetric argument shows the other inclusion, so $P_2 = gP_1g^{-1}$. Since distinct elements of $\mathcal{P}$ are not conjugate, $P_1 = P_2$. Therefore, we can replace $H_2$ with $H_1$ and see that $gH_1g^{-1} \bigcap H_1$ has finite index in $gH_1g^{-1}$ and in $H_1$, meaning $g$ is in the commensurator of $H_1$, which is $P_1$. So, $gP_1g^{-1} = P_1$. Since $g$, $P_1$, and $P_2$ were chosen arbitrarily, $\mathcal{P}$ is almost malnormal.
\end{proof}
\begin{definition}[{\cite[Chapter~3]{QCERF}}]
	A homomorphism $\phi: H \to G$ of hyperbolic groups with peripheral structures $\mathcal{D}$ and $\mathcal{P}$ \textbf{respects the peripheral structure on} $H$ if for every $D_i \in \mathcal{D}$, there is a $P_j \in \mathcal{P}$ such that $\phi(D_i)$ is conjugate into $P_j$ in $G$.
\end{definition}
\begin{definition}[{\cite[Definition~3.2]{QCERF}}]
	If $\phi: H \to G$ is a homomorphism which respects the peripheral structure on $H$, then any filling of $G$ induces a filling of $H$ as follows. For each $i$, there is some $c_i = c(D_i)$ in $G$ and some $P_{j_i} \in \mathcal{P}$ such that $\phi(D_i) \subseteq c_i P_{j_i} c_i^{-1}$. The \textbf{induced filling kernels} $K_i \trianglelefteq D_i$ are given by
\begin{align*}
	K_i = \phi^{-1}(c_i N_{j_i} c_i^{-1}) \bigcap D_i
\end{align*}
The \textbf{induced filling} is $H \twoheadrightarrow H(K_1,...,K_n)$. The map $\phi$ induces a homomorphism $\bar{\phi}: H(K_1,...,K_n) \to G(N_1,...,N_m)$.
\end{definition}
\begin{definition}[{\cite[Definition~3.3]{QCERF}}]
	Let $H$ be a quasi-convex subgroup, with peripheral structure $\mathcal{D}$, of a hyperbolic group $G$ and $\mathcal{P}$ a peripheral structure on $G$ such that the inclusion $H \xhookrightarrow{} G$ respects the peripheral structure on $H$. A filling $G \twoheadrightarrow G(N_1,...,N_m)$ is an $H$\textbf{-filling} if whenever $H \bigcap g P_i g^{-1}$ is infinite, we have $g N_i g^{-1} \subseteq s D_j s^{-1} \subseteq H$ for some $s \in H$ and $D_j \in \mathcal{D}$.\par 
	Given a collection $\mathcal{H}$ of quasi-convex subgroups, a filling is an $\mathcal{H}$\textbf{-filling} if it is an $H$-filling for each $H \in \mathcal{H}$. 
\end{definition}
The following lemma is immediate from the definitions.
\begin{lemma}\label{respect}
	Let $G$ be a hyperbolic group with a finite splitting $(\Gamma,G)$ with only two-ended edge groups. Let $\mathcal{P}$ be the peripheral structure on $G$ induced by the edge groups. For each vertex group $G_v$, let $\mathcal{D}_v$ be the peripheral structure induced by $Inc|_{G_v}$. For each edge group $G_e$, let the peripheral structure be $\{G_e\}$. Then, the following inclusions respect peripheral structures:
	\begin{itemize}
		\item each inclusion of a vertex group into $G$,
		\item each inclusion of an edge group into a vertex group,
		\item each inclusion of an edge group into $G$.
	\end{itemize}
\end{lemma}
	With \Cref{respect}, we can deduce an important lemma regarding the effect of relatively hyperbolic Dehn filling on the graph of groups decomposition of a hyperbolic group.
\begin{lemma}\label{quotient-graph}
Let $(\Gamma,G)$ be an elementary splitting of $G$ with edge groups $\mathcal{E}$. Let $\mathcal{P}$ be the peripheral structure induced by the edge groups. Then, for sufficiently long $\mathcal{E}$-fillings with respect to $\mathcal{P}$, $\pi:G \to \bar{G}$ of $G$, the quotient $\bar{G}$ has a graph of groups splitting $(\Gamma,\bar{G})$ such that for each vertex $v$, $\bar{G}_v$ is the induced filling of $G_v$ and for each edge $e$, $\bar{G}_e$ is the induced filling of $G_e$.
\end{lemma}
\begin{proof}
	We need to show two things. Firstly, there is an induced graph of groups with the same graph $\Gamma$ and each vertex group and edge group being the image of the original vertex and edge group. Secondly, the fundamental group of this induced graph of groups is indeed the quotient $\bar{G}$.\par 
	For each edge $e$ of $\Gamma$ with endpoints $v$ and $w$, there are injections $G_e \xrightarrow{\phi_v} G_v$ and $G_e \xrightarrow{\phi_w} G_w$. So, we have the following commutative diagram when we apply Dehn filling.\par
\begin{center}
\begin{tikzcd}
G_e \arrow[r, "\phi_v"] \arrow[d, "\pi|_{G_e}"] & G_v \arrow[d, "\pi|_{G_v}"] \\
\pi(G_e) \arrow[r, "\bar{\phi}_v", dotted]      & \pi(G_v)                   
\end{tikzcd}
\end{center}  
	The induced map in the bottom exists if Ker$(\pi|_{G_e}) \subseteq$ Ker$(\pi|_{G_v} \circ \phi_v)$. This is true because we are applying Dehn filling with filling kernels coming from the edge groups and they are included in the vertex groups. We see that we always get a generalized graph of groups $(\Gamma,\bar{G})$ in the sense that the edge maps from the edge groups to the vertex groups may not be injective.\par  
	In our situation, $G$ is hyperbolic, and all edge groups are two-ended (hence quasi-convex), and that implies all vertex groups are quasi-convex by \Cref{quasi}. Then, by \cite[Proposition~7.5]{AGM14}, all of these groups are fully relatively quasi-convex (see \cite[Definition~7.3, Definition~7.8]{AGM14} ). In order to apply \cite[Proposition~4.4]{QCERF} to the vertex groups, we also need to make sure that the filling is a $G_v$ filling for each vertex group $G_v$. We look at the Bass-Serre tree associated to this elementary splitting and suppose $G_v \bigcap gP_ig^{-1}$ is infinite. Let $E_i$ be the edge group containing $N_i$. Since $N_i$ has finite index in $E_i$, we have that $E_i'=G_v \bigcap gE_ig^{-1}$ is infinite. Let $e_i$ be the edge fixed by $E_i$ in the Bass-Serre tree. Then, $gE_ig^{-1}$ is the stabilizer of $ge_i$. Thus, $E_i'$ fixes the arc from $v$ to the endpoint of $ge_i$ farthest to $v$. Let $e$ be the edge on this arc with endpoint $v$, then $E_i'$ is an infinite subset of its stablizer $E$. Since these edge stabilzers are two-ended, having an infinite intersection means they induce the same peripheral group in $G$, so by ensuring the filling kernels are $\mathcal{E}$-filling kernels, we've made sure that $N_i$ conjugates into $E\leq G_v$. Since we used incident edge groups to induce the peripheral structure on $G_v$, there is some $D_j$ in the peripheral structure of $G_v$ and some element $s$ of $G_v$ such that $E \leq sD_js^{-1}$. This means we are taking $G_v$-fillings for each $G_v$. So, we can apply \cite[Proposition~4.4]{QCERF} to see that the induced edge maps must be injections under sufficiently long fillings and the vertex groups and edge groups are exactly induced fillings of the original vertex and edge groups. In other words, $\pi(G_v) = G_v(N_1,N_2,...,N_{k_v})$ with $N_1,...,N_{k_v}$ the filling kernels involved in edge groups incident to $v$, and $\pi(G_e) = G_e/N_e$.\par 
	To see that $\bar{G} \cong \pi_1(\Gamma,\bar{G})$, we can follow the proof of \cite[Lemma~8.1]{AGM14}. We don't need the fully $\mathcal{P}$-elliptic condition since we are only working on level one of their hierarchy and we induced peripheral structures on each vertex group and $G$ itself using the edge groups. So, the fact that elements of kernels lie in the right places for their proof to work is ensured by our set-up.
\end{proof}
\begin{definition}
Let $G$ be a group. Let $\mathcal{A}$ be a family of subgroups of $G$ which is closed under taking subgroups and conjugations. Let $\mathcal{H}$ be a family of subgroups of $G$. A \textbf{$(G,\mathcal{A},\mathcal{H})$-tree} is a tree $T$ together with an action of $G$ such that each edge stabilizer is in $\mathcal{A}$ and each group in $\mathcal{H}$ fixes a point in $T$. Usually, the context is clear so we omit $G$ (or even the whole prefix) and just say $(\mathcal{A},\mathcal{H})$-tree (or tree).\\
A subgroup of $G$ is \textbf{universally elliptic} (with respect to $(\mathcal{A},\mathcal{H})$) if it fixes a point in each $(\mathcal{A},\mathcal{H})$-tree.\\
An $(\mathcal{A},\mathcal{H})$-tree is \textbf{universally elliptic} if all of its edge stabilizers are universally elliptic.
\end{definition}
\begin{definition}
	A \textbf{JSJ $(G,\mathcal{A},\mathcal{H})$-tree} $T$ is an $(G,\mathcal{A},\mathcal{H})$-tree such that:
\begin{itemize}
	\item $T$ is universally elliptic with respect to $(G,\mathcal{A},\mathcal{H})$, and
	\item for every universally elliptic $(G,\mathcal{A},\mathcal{H})$-tree $T'$, there is a $G$-equivariant map $f: T \to T'$
\end{itemize}
The induced graph of groups of a JSJ tree is a \textbf{JSJ decomposition}. \\
A vertex stabilizer of a JSJ tree is \textbf{flexible} if it is not universally elliptic.
\end{definition}
\begin{definition}
	A subgroup $Q \leq G$ is \textbf{quadratically hanging (QH)} if 
\begin{itemize}
	\item $Q$ is a vertex stabilizer in an $(\mathcal{A},\mathcal{H})$-tree $T$ of $G$;
	\item $Q$ is an extension $1 \to F \to Q \to \pi_1^{orb}(\Sigma) \to 1$ where $\Sigma$ is a compact hyperbolic 2-orbifold; we call $F$ the \textbf{fiber}, and $\Sigma$ the \textbf{underlying manifold};
	\item The image of each incident edge stabilizer, and each intersection $Q \bigcap gHg^{-1}$ for $H \in \mathcal{H}$  is a finite subgroup of $\pi_1^{orb}(\Sigma)$ or a subgroup of $\pi_1^{orb}(B)$ where $B$ is a boundary component of $\Sigma$.
\end{itemize}
\end{definition}
	We refer the readers to \cite{JSJ} for the details on orbifolds. We simply include the following fact here.
\begin{proposition}[{\cite[Proposition~5.2]{JSJ}}]\label{Euler}
	A 2-orbifold is hyperbolic if and only if its Euler characteristic is negative. 
\end{proposition}
	We give the definition of QH vertex groups here since they appear in the canonical JSJ decomposition given by Bowditch. Also, the fact that they fit into a specific kind of short exact sequence guarantees their residual finiteness (\Cref{GOOD}). It is important to keep their images in the Dehn filling residually finite.\par 
	We point out that Bowditch, in \cite{Bow98}, used the word ``cocompact Fuchsian" to mean ``finite-by-cocompact Fuchsian". A group $G$ is \textbf{cocompact Fuchsian} if it is a discrete subgroup of $Isom(\mathbb{H}^2)$ and there is a compact $K \subseteq \mathbb{H}^2$ such that $G \cdot K = \mathbb{H}^2$.
\begin{theorem}[{\cite[Theorem~5.28,Theorem~6.2]{Bow98}}]\label{JSJ}
	Let $G$ be a one-ended hyperbolic group that is not finite-by-cocompact Fuchsian, then $G$ has a canonical JSJ decomposition where the edge groups are two-ended and quasi-convex and the vertex groups are one of the following types:
\begin{itemize}
	\item[1)] maximal two-ended
	\item[2)] QH with finite fiber
	\item[3)] not the first two types
\end{itemize}
Furthermore, type 2) and 3) vertices are never adjacent to a type 2) or 3) vertex, and type 1) vertices are never adjacent to a type 1) vertex.
\end{theorem}
\begin{remark}
	By \cite[Theorem~9.18]{JSJ} and the sentence following that theorem in \cite{JSJ}, vertex groups of type 3) are universally elliptic.
\end{remark}
\begin{lemma}\label{GOOD}
	If there is a short exact sequence $1 \to F \to Q \to \pi_1^{orb}(\Sigma) \to 1$ where $F$ is finite, $Q$ hyperbolic and $\Sigma$ is a compact hyperbolic 2-orbifold, then $Q$ is residually finite.
\end{lemma}
\begin{proof}
	This is well-known. Fundamental groups of 2-orbifolds are good in the sense of Serre (\cite[Proposition~3.4]{GOODBianchi}), and any extension of a residually finite group by a good group is residually finite. For details, see \cite[Chapter~7]{PPDG}.
\end{proof}
	We would like the quotient of these QH vertex groups to fit into another short exact sequence of the same type so we can apply \Cref{GOOD} again. \Cref{quotient-seq} shows that there is some short exact sequence the quotient group fits in.
\begin{lemma}\label{quotient-seq}
	Let $1 \to F \to G \xrightarrow{\phi} H \to 1$ be a short exact sequence of groups with $F$ finite, $G$ hyperbolic relative to a collection $\mathcal{P} = \{P_1,...,P_k\}$ of finitely generated infinite subgroups. Then, for sufficiently long fillings $\bar{G} = G(N_1,...,N_k)$ of $G$ with respect to $\mathcal{P}$, we have a short exact sequence $1 \to F \to \bar{G} \to \bar{H} \to 1$, where $\bar{H} = H(\phi(N_1),...,\phi(N_k))$.
\end{lemma}
\begin{proof}
	Let $F$ be the finite set that we want to inject into the filled group in the relative hyperbolic Dehn filling theorem. Then, for sufficiently long fillings, $F$ injects into $\bar{G}$ via the quotient map. In the quotient $\bar{G}$, $F$ is still normal so we can talk about the following diagram:\\
\begin{center}
\begin{tikzcd}
1 \arrow[r] & F \arrow[r, "i"] \arrow[d, "p|_F"] & G \arrow[r, "\phi"] \arrow[d, "p"] & H \arrow[r] \arrow[d, "q", dashed] & 1 \\
1 \arrow[r] & p(F) \arrow[r]                        & \bar{G} \arrow[r, "\psi"]          & \bar{G}/p(F) \arrow[r]                & 1
\end{tikzcd}
\end{center}
Here, we define $q: H \to \bar{G}/p(F)$ to be the homomorphism that makes the diagram commute.\\
\begin{itemize}
	\item $q$ well-defined\\
	Let $g,h \in G$ such that $\phi(g) = \phi(h)$. This means $gF = hF$, so $q(\phi(g)) = \psi(p(g)) = p(g)p(F) = p(gF) = p(hF) = p(h)p(F) = \psi(p(h)) = q(\phi(h))$ ($F$ comes out of $p$ because the left square commutes).
	\item $q$ surjective\\
	Let $x \in \bar{G}/p(F)$, then since $\psi$ and $p$ are quotient maps, exists $g \in G$ such that $\psi(p(g)) =x$. Thus, $q(\phi(g))=x$.\\
	\item $ker(q) \supseteq \langle \langle \phi (N_1),...,\phi (N_k) \rangle \rangle$ \\
	Let $h \in \phi(N_1)$, then we have $g \in N_1$ such that $\phi(g)=h$. Thus, $q(h) = \psi (p(g)) = \psi (1) = 1$, meaning $N_1 \subseteq ker(q)$. Similarly, all other images of $N_i$ are in the kernel. Therefore, $ker(q) \supseteq \langle \langle \phi (N_1),...,\phi (N_k)\rangle \rangle$.\\
	 \item $ker(q) \subseteq \langle \langle \phi (N_1),...,\phi (N_k) \rangle \rangle$\\
	 Suppose that there is a $h \in ker(q) - \langle \langle \phi (N_1),...,\phi (N_k) \rangle \rangle$. Let $g \in G$ such that $\phi(g) = h$. Since $\phi(\lr{\lr{N_1,...,N_k}}) \subseteq \lr{\lr{\phi(N_1),...,\phi(N_k)}}$, the element $g$ must not lie in $\lr{\lr{N_1,...,N_k}}$. Thus, $p(g) \neq 1$. Also, $\psi(p(g)) = q(\phi(g)) = q(h) = 1$ by assumption, so $p(g) \in p(F) - \{1\}$. The fact that $p(g) \in p(F)$ implies that there is an element $r \in F$ such that $r^{-1}g \in \langle \langle N_1,...,N_k \rangle \rangle$, so $g \in F \langle \langle N_1,...,N_k \rangle \rangle$ (this is a group since both factors are normal). This further implies that $h=\phi(g)\in \langle \langle \phi (N_1),...,\phi (N_k) \rangle \rangle$, contradiction.
\end{itemize}
The four bullet points above allow us to use the first isomorphism theorem and conclude that $\bar{H} \cong \bar{G}/p(F)$ naturally and fits in the diagram at the right place. The second row of the diagram is then our conclusion.
\end{proof}
	Now, we want to show that the $\bar{H}$ we obtain in \Cref{quotient-seq} can be made to be a 2-orbifold fundamental group, so we can apply \Cref{GOOD}. This is the content of \Cref{orbi-shift}.
\begin{definition}
	Given a boundary subgroup $B$ of a 2-orbifold, the \textbf{essential subgroup} of $B$ is the maximal cyclic subgroup of $B$.
\end{definition}
\begin{remark}
	The boundary of a 2-orbifold is either a circle or an arc bounded by two mirrors. With a circle boundary, the boundary subgroup is $\mathbb{Z}$ so its essential subgroup is itself. For the arc boundary, the boundary subgroup is $D_\infty = \mathbb{Z}/2\mathbb{Z} * \mathbb{Z}/2\mathbb{Z}$. Its essential subgroup is the unique index-two cyclic subgroup.
\end{remark}
\begin{lemma}\label{orbi-shift}
	Let $Q$ be a group with short exact sequence $1 \to F \to Q \to \pi_1^{orb}(\Sigma) \to 1$ where $F$ is finite and $\Sigma$ is a compact hyperbolic 2-orbifold. Let $\mathcal{H}$ be the collection of preimages of the boundary subgroups of $\Sigma$. Let $\mathcal{D} = \{D_1,...,D_n\}$ be the peripheral structure on $Q$ induced by $\mathcal{H}$. Then, for $H$-filling kernels $(N_1,...,N_n)$ with large enough indices that map into essential subgroups of the boundary groups, $Q(N_1,...,N_n)$ fits into a short exact sequence $1 \to F \to \bar{Q} \to \pi_1^{orb}(\Sigma_Q') \to 1$ where $\Sigma_Q'$ is a compact hyperbolic 2-orbifold.
\end{lemma}
\begin{proof}
	By \Cref{quotient-seq}, for sufficiently long fillings, there is a short exact sequence $1 \to F \to \bar{Q} \to \bar{Q}/F \to 1$ where $\bar{Q}/F = \pi_1^{orb}(\Sigma)/ \lr{\lr{\phi(N_1),...,\phi(N_n)}}$. We realize this quotient of the orbifold fundamental group as the orbifold fundamental group of a new 2-orbifold $\Sigma'$. \\
In the first case where the boundary is a circle, $\phi(N_i) = q_i\mathbb{Z}$ and the quotient at the boundary is $\mathbb{Z}/q_i\mathbb{Z}$. We can realize this by gluing a disk with a conical point of order $q_i$ to the boundary circle. \\

\begin{adjustbox}{center}

\tikzset{every picture/.style={line width=0.75pt}} %set default line width to 0.75pt        

\begin{tikzpicture}[x=0.75pt,y=0.75pt,yscale=-1,xscale=1]

%Shape: Ellipse [id:dp52247819775824] 
\draw  [color={rgb, 255:red, 0; green, 0; blue, 0 }  ,draw opacity=1 ] (181.5,143.97) .. controls (170.43,143.78) and (160.79,127.97) .. (159.97,108.66) .. controls (159.14,89.34) and (167.44,73.84) .. (178.5,74.03) .. controls (189.57,74.22) and (199.21,90.03) .. (200.03,109.34) .. controls (200.86,128.66) and (192.56,144.16) .. (181.5,143.97) -- cycle ;
%Curve Lines [id:da7845503629820452] 
\draw    (114.8,75) .. controls (153.8,71) and (129.8,89) .. (178.5,74.03) ;
%Curve Lines [id:da9346829056337393] 
\draw    (122.8,143) .. controls (164.1,130.03) and (152.8,145) .. (181.5,143.97) ;
%Straight Lines [id:da029376299692705388] 
\draw    (235,101) -- (331.8,101.98) ;
\draw [shift={(333.8,102)}, rotate = 180.58] [color={rgb, 255:red, 0; green, 0; blue, 0 }  ][line width=0.75]    (10.93,-3.29) .. controls (6.95,-1.4) and (3.31,-0.3) .. (0,0) .. controls (3.31,0.3) and (6.95,1.4) .. (10.93,3.29)   ;
%Curve Lines [id:da5369960333299064] 
\draw    (353.8,76) .. controls (392.8,72) and (368.8,90) .. (417.5,75.03) ;
%Curve Lines [id:da014626731161976991] 
\draw    (361.8,144) .. controls (403.1,131.03) and (391.8,146) .. (420.5,144.97) ;
%Straight Lines [id:da43564650310710884] 
\draw    (417.5,75.03) -- (502.8,105) ;
%Straight Lines [id:da9149109349532394] 
\draw    (420.5,144.97) -- (502.8,105) ;
%Curve Lines [id:da940249264291561] 
\draw    (417.5,75.03) .. controls (395.8,86) and (398.8,152) .. (420.5,144.97) ;
%Curve Lines [id:da626837820631412] 
\draw  [dash pattern={on 4.5pt off 4.5pt}]  (417.5,75.03) .. controls (447.8,77) and (441.8,147) .. (420.5,144.97) ;

% Text Node
\draw (510,95) node [anchor=north west][inner sep=0.75pt]   [align=left] {$q_i$};
\end{tikzpicture}
\end{adjustbox}

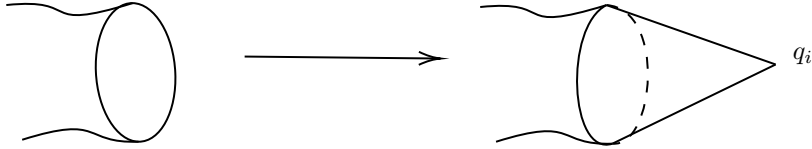
\centerline{\textbf{FIGURE 1}. CIRCLE BOUNDARY}\leavevmode
\par 
By assumption, the image of $N_i$ that maps to $\mathbb{Z}/2\mathbb{Z}*\mathbb{Z}/2\mathbb{Z}$ is a finite index subgroup the essential subgroup. The quotient at the boundary is a finite dihedral group $\mathcal{D}_{r_i}$ for some $r_i \in \mathbb{N}$. This can be realized by extending the surface past the boundary arc so the two mirrors meet at an angle of $2\pi/r_i$ (consequently creating a corner reflector of order $r_i$).\\

\tikzset{every picture/.style={line width=0.75pt}} %set default line width to 0.75pt        

\begin{adjustbox}{center}

\tikzset{every picture/.style={line width=0.75pt}} %set default line width to 0.75pt        

\begin{tikzpicture}[x=0.75pt,y=0.75pt,yscale=-1,xscale=1]

%Straight Lines [id:da1801707535718704] 
\draw [color={rgb, 255:red, 74; green, 144; blue, 226 }  ,draw opacity=1 ]   (206.87,69.87) -- (206.87,116.63) ;
%Straight Lines [id:da49785077878110107] 
\draw    (206.87,116.63) -- (314.04,115.33) ;
%Straight Lines [id:da6707502170040758] 
\draw [color={rgb, 255:red, 74; green, 144; blue, 226 }  ,draw opacity=1 ]   (314.04,115.33) -- (314.04,71.17) ;
%Straight Lines [id:da5494977921376571] 
\draw    (342.44,94.55) -- (384.38,93.31) ;
\draw [shift={(386.38,93.25)}, rotate = 178.31] [color={rgb, 255:red, 0; green, 0; blue, 0 }  ][line width=0.75]    (10.93,-3.29) .. controls (6.95,-1.4) and (3.31,-0.3) .. (0,0) .. controls (3.31,0.3) and (6.95,1.4) .. (10.93,3.29)   ;
%Straight Lines [id:da11108220100023791] 
\draw [color={rgb, 255:red, 0; green, 0; blue, 0 }  ,draw opacity=1 ]   (410.77,97.15) -- (461.41,95.85) ;
%Straight Lines [id:da4337677242578355] 
\draw [color={rgb, 255:red, 74; green, 144; blue, 226 }  ,draw opacity=1 ]   (409.16,68.57) -- (410.77,97.15) ;
%Straight Lines [id:da2227617564145996] 
\draw [color={rgb, 255:red, 74; green, 144; blue, 226 }  ,draw opacity=1 ]   (461.41,71.17) -- (461.41,95.85) ;
%Curve Lines [id:da4366179579330116] 
\draw [color={rgb, 255:red, 74; green, 144; blue, 226 }  ,draw opacity=1 ]   (461.41,95.85) .. controls (465.43,121.83) and (437.29,116.63) .. (438.63,130.92) ;
%Curve Lines [id:da6318532278315121] 
\draw [color={rgb, 255:red, 74; green, 144; blue, 226 }  ,draw opacity=1 ]   (410.77,97.15) .. controls (407.82,124.43) and (438.63,114.03) .. (438.63,130.92) ;

% Text Node
\draw (187.3,45.14) node [anchor=north west][inner sep=0.75pt]   [align=left] {Mirror};
% Text Node
\draw (294.48,50.33) node [anchor=north west][inner sep=0.75pt]   [align=left] {Mirror};
% Text Node
\draw (433.07,130.11) node [anchor=north west][inner sep=0.75pt]    {$r_{i}$};

\end{tikzpicture}

\end{adjustbox}

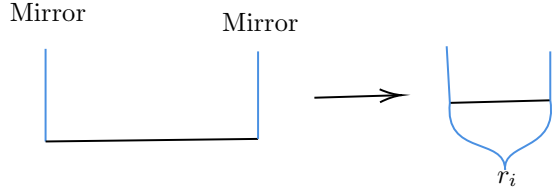
\centerline{\textbf{FIGURE 2}. ARC BOUNDARY}\leavevmode
\par
By applying the two above constructions, we obtain a new compact 2-orbifold $\Sigma'$ which $\Sigma$ embeds in. Adding a disk with a cone point of order $q_i$ increases the Euler characteristic by $1/q_i$, so we can make each of these arbitrarily small by making sure the index of $N_i$ is large enough. Each corner reflector of order $r_i$ contributes $\frac{1}{2}(\frac{1}{r_i}-1)$ to the Euler characteristic, which is negative. Since we started with an orbifold of negative Euler characteristic, we can make sure that  $\Sigma'$ has negative Euler characteristic, and thus hyperbolic by \Cref{Euler}.\par 
Now, by Van Kampen's theorem, $\pi_1(\Sigma') \cong \pi_1^{orb}(\Sigma)/\lr{\lr{\phi(N_1),...,\phi(N_k)}}$.
\end{proof}
	The JSJ decomposition of $G$ may have multiple QH vertex groups. Fixing one QH vertex group $Q$, \Cref{orbi-shift} says that $Q$ has residually finite images in some cofinal family of fillings of $G$. Now, we want to show that these families for different QH vertex groups intersect to give another cofinal family. In other words, there are plenty of quotients where all QH vertices have residually finite images.
\begin{lemma}\label{QH}
	Let $G$ be a one-ended hyperbolic group and $(\Gamma,G)$ its canonical JSJ-decomposition. Let $\mathcal{P}$ be a peripheral structure induced by the edge groups, $\mathcal{E}$, of $(\Gamma,G)$. For each vertex group $G_v$, induce a peripheral structure $D_v$ using $Inc|_{G_v}$. Then, there is a cofinal family of finite index filling kernels for $\mathcal{P}$ such that the induced filling kernels on each QH vertex satisfies the conditions of \Cref{orbi-shift}.
\end{lemma}
\begin{proof}
	Let $F_1,...,F_k$ be the collection of fibers of QH vertex groups. Let $F = F_1 \bigcup ... \bigcup F_k$ be the set we want to embed into the filling. By \Cref{quotient-graph}, for sufficiently long $\mathcal{E}$-fillings, the filled group has a splitting with the same graph $\Gamma$ with each vertex group and edge group corresponding quotients.
	Take an arbitrary QH vertex group $Q$ with fiber $F_Q$ and quotient map $\psi: Q \twoheadrightarrow Q/F_Q = \pi_1^{orb}(\Sigma_Q)$. We would like to use \Cref{orbi-shift}. However, we are filling the big group $G$ now, so we want to show that we can induce the fillings realized in \Cref{orbi-shift} from the fillings of $G$. Let $\mathcal{D} = \{D_1,...,D_n\}$ be a peripheral structure induced by incident edge groups on $Q$. The filling of $G$ induces a filling of $Q$ by \cite[Proposition~4.4]{QCERF}.\par 
	Firstly, we show that filling kernels within $Q$ that satisfies the requirements of \Cref{orbi-shift} are plenty. The only concern is the arc boundary. We would like kernels $N_i$ whose image under $\psi$ is inside the essential subgroup. We start with any kernel $N_i$ with large enough index so we can ensure \Cref{quotient-seq} holds. Now, $N_i$ is infinite and $F$ is finite so $\psi(N_i)$ is infinite. Therefore, $\psi(N_i)$ will have finite index in the boundary subgroup since the boundary subgroup is two-ended. So, the intersection of $\psi(N_i)$ with the essential subgroup (call it $E$) has finite index in $E$. We look at the preimage of this intersection, $\psi^{-1}(\psi(N_i) \bigcap E)$, which is infinite, so it has finite index in $D_i$ as $D_i$ is two-ended. Therefore, we can get a finite index normal subgroup of $D_i$ by passing to a subgroup of $\psi^{-1}(\psi(N_i) \bigcap E)$. Note, this is also a subgroup of $N_i$. Thus, starting from a nice enough filling kernel $N_i$ that avoids some bad finite set $B$, we can obtain the correct filling kernel for the requirements of \Cref{orbi-shift} by passing to a subgroup of $N_i$. It will also avoid $B$. Combining this process for all $D_i$, we obtain a cofinal family of filling kernels $S'$ for \Cref{orbi-shift} to work on $Q$. Next, we show that we can induce these kernels from filling $G$.\par 
	Take $(K_1,...,K_n) \in S'$. Then, $K_1 \trianglelefteq D_1 \subseteq c_1P_{j_1}c_1^{-1}$ for some $c_1 \in G$ and $P_{j_1} \in \mathcal{P}$. Now, $K_1$ has finite index in two-ended group $D_1$, so it is infinite, and $P_{j_1}$ is also two-ended, so $c_1^{-1}K_1c_1$ has finite index in it. Thus, by passing to a further finite index subgroup to ensure normality, we get a candidate filling kernel $N_{1,j_1} \trianglelefteq P_{j_1}$. Do the same for all $K_i$ and taking intersections if multiple candidates land in the same $P_j$. For each $P_j$ that is not involved after the process ends, we can pick any filling kernel that is allowed in $C_{B_1}$. Now, the induced filling of $Q$ works as desired. Now, for each choice of filling kernels in $S'$, we obtain finite index filling kernels for $(G,\mathcal{P})$. This shows that we get a cofinal family of filling kernels for each QH vertex so \Cref{orbi-shift} applies. This family is cofinal since we can start with some filling kernels in $\mathcal{P}$ avoiding some finite set $V$ and then check the induced filling kernels induced on $Q$. Go through the process in the previous paragraph to get filling kernels in $S'$. Now, the process in this paragraph gives us desirable filling kernels in $C_V$.\par
	Now, we have picked a few cofinal families of filling kernels $S_1,...,S_k$, each making one of the QH vertex retain a nice short exact sequence in the quotient. It remains to show that $S = S_1 \bigcap S_2 \bigcap ... \bigcap S_k$ is a cofinal family of filling kernels. Let $(N_{i,1},...N_{i,n}) \in S_i$ and consider the kernels in $P_1$. By our construction, $N_{i,1}$ are normal with finite index in $P_1$, so we simply take their intersection and get a normal subgroup of finite index $M_1 \trianglelefteq P_1$. The image of $M_1$ in each orbifold group associated to the QH vertices where a conjugate of $P_1$ is attached will be a subgroup of the essential subgroup after passing to a further finite index if necessary. The same procedure gives us $(M_1,...,M_n)$, and this is an element in $S$. We have described a procedure to produce an element of $S$ from a choice of elements of $S_1,...,S_k$. Given some finite set $D$, each $S_i$ intersects $C_D$ nontrivially, so we can pick one element from each $S_i$. Now, going through the procedure we just described, we get an element of $S$. Since the procedure only involves intersection, so the new element we created is still in $C_D$, meaning $S \bigcap C_D \neq \emptyset$. 
\end{proof}
\section{Main Theorem}
\Tech*
\begin{proof}[Proof of \Cref{tech}]
	Let $G$ be a one-ended hyperbolic group that is not finite-by-cocompact Fuchsian. By \Cref{JSJ}, there is a finite graph of groups $(\Gamma,G)$ with three kinds of vertices. Let $\mathcal{E}$ be the collection of edge groups. Let $\mathcal{P} = \{P_1,...,P_m\}$ be a peripheral structure induced by the edge groups and $F_1,...,F_k$ be the collection of fibers of QH vertex groups. Let $F = \{1,g\} \bigcup F_1 \bigcup ... \bigcup F_k$ be the set we want to embed into the filling. By \Cref{quotient-graph}, for sufficiently long $\mathcal{E}$-fillings, the filled group has a splitting with the same graph $\Gamma$ with each vertex group and edge group corresponding quotients. Notice that for any finite index fillings, the two-ended vertex groups become finite, so do the type 1) vertex groups. In this paragraph, from \Cref{fill}, we obtain a finite bad set $B_1$ such that all desired properties described in this paragraph are present in the quotient if we pick filling kernels from $C_{B_1}$.\par 
	By \Cref{QH}, we have a cofinal familiy of finite index filling kernels $S$ making all QH-vertices satisfy the requirements of \Cref{orbi-shift}. Then, \Cref{GOOD} ensures the quotient of the QH groups are residually finite.\par 
	Thus, we have some $C_{B_1}$ with finite $B_1 \subseteq G \setminus \{1\}$, that gives fillings where two-ended vertex groups becomes finite. We also have a cofinal family of finite index filling kernels $S$ making all QH-vertices residually finite in the fillings. Of course, in all of these fillings, the image of $g$ is nontrivial. Let $T=S \bigcap C_{B_1} \neq \emptyset$ (the intersection is non-empty by definition). Take any finite subset $E \subseteq G \setminus \{1\}$. Then, $T \bigcap C_E = S \bigcap C_{B_1} \bigcap C_E = S \bigcap C_{{B_1}\bigcup E} \neq \emptyset$. So, $T$ is a cofinal family of finite index filling kernels that will give fillings satisfying all three properties we propose in the statement of the theorem.
\end{proof}
	Conditions in the following proposition matches up with the conditions in \cite{Das22} and allow us to utilize \cite[Theorem~1.2]{Das22}.
\begin{proposition}[{\cite[Chapter~2,Corollary~4.14]{JSJ}}]
	A vertex group $G_v$ in a JSJ $(G,\mathcal{A},\mathcal{H})$-tree is a universally elliptic vertex group if and only if it fixes a point on every $(G,\mathcal{A},\mathcal{H}\bigcup Inc|_v)$-tree.
\end{proposition}

\begin{remark}
We note that ``elementary" in the following theorem of Dasgupta's means a larger collection of edge groups than our ``elementary" since the other author is interested in elementary subgroups in a relatively hyperbolic group. However, by simply specializing to our situation where the group is hyperbolic and the peripheral structure only contains two-ended groups, the two notions are the same. Also, by taking finite index filling kernels in our situation, the filling is $\mathcal{M}$-finite in the sense of the following theorem. 
\end{remark}

\begin{theorem}[{\cite[Theorem~1.2]{Das22}}]\label{rigid}
	Let $(G,\mathcal{P})$ be a relative hyperbolic group pair that admits no nontrivial elementary splitting relative to $\mathcal{P}$, then for sufficiently long $\mathcal{M}$-finite fillings $\bar{G}$ relative to $\mathcal{P}$, $\bar{G}$ admits no nontrivial elementary splitting relative to $\bar{\mathcal{P}}$.
\end{theorem}
\begin{proposition}\label{cor}
	Let $G$ be a one-ended hyperbolic group and $g\in G$ nontrivial. Then, there is a homomorphism $\phi:G \to \bar{G}$ such that $\phi(g)$ is nontrivial and $\bar{G}$ admits a finite graph of groups decomposition with finite edge groups and vertex groups are either residually finite or rigid hyperbolic groups.
\end{proposition}
\begin{proof}
	If $G$ is a finite-by-cocompact Fuchsian group, then we take $\bar{G} = G$ with the trivial graph of groups decomposition since $G$ is already residually finite. So, let $G$ be a one-ended hyperbolic group that is not finite-by-cocompact Fuchsian. By \Cref{JSJ}, there is a finite graph of groups with three kinds of vertices. Let $P_1,...,P_n$ be the peripheral structure induced by the edge groups. Now, applying \Cref{tech} with respect to $P_1,...,P_n$, there is a cofinal family of finite index fillings $\bar{G}_n$ such that the conclusions of \Cref{tech} holds. The universally elliptic vertex groups, by \Cref{rigid}, for sufficiently long finite index fillings, become hyperbolic groups that admits no nontrivial splitting relative to the incident edge groups, but now all edge groups are finite, so they act elliptically on the associated Bass-Serre tree if there is any splitting (relative to an empty collection), so indeed the universally elliptic vertex groups become rigid after passing to the quotient. These two effects can happen in the same quotient by the definition of a cofinal family of finite index filling kernels. 
\end{proof}
	The following theorem of Dunwoody's reduces the case of infinite-ended hyperbolic groups to one-ended case.
\begin{theorem}[{\cite[Theorem~5.1]{Dun85}}]\label{access}
	A finitely presented group $G$ has more than one end if and only if $G$ admits a finite graph of groups decomposition with finite edge groups and no infinite-ended vertex groups.
\end{theorem}
\Main*

\begin{proof}[Proof of \Cref{main}]
	We prove the contrapositive: if all rigid hyperbolic groups are residually finite, then all hyperbolic groups are residually finite. \par 
	Let $G$ be a hyperbolic group. It is easy to see that the zero-ended and two-ended hyperbolic groups are residually finite. For the infinite-ended case, we apply \Cref{access} and \Cref{quasi} to see that it splits into a finite graph of groups with finite edge groups and hyperbolic vertex groups that are not infinite-ended. By \Cref{RF}, we reduce to checking one-ended hyperbolic groups that are not finite-by-cocompact Fuchsian groups (finite-by-cocompact Fuchsian groups are residually finite by \Cref{GOOD}). \par 
	Let $G$ be a one-ended hyperbolic group that is not finite-by-cocompact Fuchsian and $g \in G$ nontrivial, then \Cref{cor} says that we have a map $\phi:G \to \bar{G}$ with $\phi(g)$ nontrivial and $\bar{G}$ satisfies the conditions of \Cref{RF} by our assumption that all rigid hyperbolic groups are residually finite. Thus, $\bar{G}$ is residually finite, so there is a further map $\psi$ to a finite group where $\psi \circ \phi (g)$ stays nontrivial. So, $G$ is residually finite.
\end{proof}
\newpage
\printbibliography

\end{document}